\documentclass[a4paper, 12pt]{article}
\usepackage{}

\usepackage{mathrsfs}
\usepackage{epsfig}
\usepackage{amsmath}
\usepackage{amssymb}
\usepackage{latexsym}
\usepackage{amsfonts}
\usepackage{amsthm}

\usepackage[numbers,sort&compress]{natbib}
\usepackage{graphicx}
\ifx\pdfoutput\undefined  \else
 \fi

\usepackage[centerlast]{caption2}

\usepackage{color}

\usepackage{txfonts}
\usepackage{amssymb}
\usepackage{mathrsfs}
\usepackage{amsfonts}

\newtheorem{theorem}{Theorem}[section]
\newtheorem{lemma}[theorem]{Lemma}
\newtheorem{cor}[theorem]{Corollary}
\newtheorem{prop}[theorem]{Proposition}
\newtheorem{conj}[theorem]{Conjecture}

\newtheorem{definition}[theorem]{Definition}

\usepackage{latexsym,amssymb}
\usepackage{amsmath}

\begin{document}

\title{Erd\H{o}s-Ginzburg-Ziv theorem for finite commutative semigroups}
\author{
Sukumar Das Adhikari$^a$\thanks{Email :
adhikari@hri.res.in} \ \ \ \ \ \ Weidong Gao$^b$\thanks{Email :
wdgao1963@aliyun.com} \ \ \ \ \ \ Guoqing Wang$^c$\thanks{Corresponding author's email: gqwang1979@aliyun.com}\\
{\small $^a$Harish-Chandra Research Institute,  Chhatnag Road, Jhusi, Allahabad 211 019, India}\\
{\small $^b$Center for Combinatorics, LPMC-TJKLC, Nankai University, Tianjin 300071, P. R. China}\\
{\small $^c$Department of Mathematics, Tianjin Polytechnic University, Tianjin, 300387, P. R. China}\\
}
\date{}
\maketitle

\begin{abstract}
Let $\mathcal{S}$ be a finite additively written
 commutative semigroup, and let $\exp(\mathcal{S})$ be its exponent which is
defined as the least common multiple of all periods of the elements
in $\mathcal{S}$. For every sequence $T$ of elements in
$\mathcal{S}$ (repetition allowed), let $\sigma(T) \in \mathcal{S}$
denote the sum of all terms of $T$. Define the Davenport constant
$\mathsf D(\mathcal{S})$ of $\mathcal{S}$ to be the least positive integer
$d$ such that every sequence $T$ over $\mathcal{S}$ of length at
least $d$ contains a proper subsequence $T'$ with
$\sigma(T')=\sigma(T)$, and define $\mathsf E(\mathcal{S})$ to be the least
positive integer $\ell$ such that every sequence $T$ over
$\mathcal{S}$ of length at least $\ell$ contains a subsequence $T'$
with
$|T|-|T'|=\left\lceil\frac{|\mathcal{S}|}{\exp(\mathcal{S})}\right\rceil\exp(\mathcal{S})$
and $\sigma(T')=\sigma(T)$.   When $\mathcal{S}$ is a finite abelian
group, it is well known that
$\left\lceil\frac{|\mathcal{S}|}{\exp(\mathcal{S})}\right\rceil\exp(\mathcal{S})=|\mathcal{S}|$
and $\mathsf E(\mathcal{S})=\mathsf D(\mathcal{S})+|\mathcal{S}|-1$. In this paper
we investigate whether $\mathsf E(\mathcal{S})\leq
\mathsf D(\mathcal{S})+\left\lceil\frac{|\mathcal{S}|}{\exp(\mathcal{S})}\right\rceil
\exp(\mathcal{S})-1$ holds true for all finite commutative
semigroups $\mathcal{S}$. We provide a positive answer to the
question above for some classes of finite commutative semigroups,
including group-free semigroups, elementary semigroups, and
archimedean semigroups with certain constraints.
\end{abstract}

\section {Introduction}

Zero-Sum Theory is a rapidly growing subfield of Combinatorial and
Additive Number Theory. The main objects of study are sequences of
terms from an abelian group (see \cite{GaoGeroldingersurvey} for a
survey, or \cite{Grynkiewicz-book} for a recent monograph). Pushed
forward by a variety of applications the last years have seen (at
least) three substantially new directions:

 (a) The study of weighted
zero-sum problems in abelian groups.

(b) The study of zero-sum
problems in (not necessarily cancellative) commutative semigroups.

(c) The study of zero-sum problems (product-one problems) in
non-commutative groups (see \cite{Bass2007, GL2010, GGr,
Grynkiewicz2013}).

In this paper we shall focus on direction (b).
 Let $G$ be an additive finite abelian group.
A sequence $T$ of elements in $G$ is called a {\sl zero-sum
sequence} if the sum of all terms of $T$ equals to zero, the
identify element of $G$. Investigations on zero-sum problems were
initiated by pioneering research on two themes, one of which is the
following result obtained in 1961 by P. Erd\H{o}s, A. Ginzburg and
A. Ziv.

\noindent\textbf{Throrem A.}  \cite{EGZ} \ (Erd\H{o}s-Ginzburg-Ziv)
\ {\sl Every sequence of $2n-1$ elements in an additive finite
abelian group of order $n$ contains a zero-sum subsequence of length
$n$. }

Another starting point is the study on Davenport constant $\mathsf D(G)$
(named after H. Davenport) of a finite abelian group $G$, which is
defined as the smallest integer $d$ such that, every sequence $T$ of
$d$ elements in $G$ contains a nonempty zero-sum subsequence. Though
attributed to H. Davenport who proposed  the study of this constant
in 1965, K. Rogers \cite{rog1} had first studied it in 1962 and this
reference was somehow missed out by most of the authors in this
area.

Let $G$ be a finite abelian group. For every integer $\ell$ with $\exp(G)\mid \ell$, let $\mathsf s_{\ell}(G)$ denote the least integer $d$ such that every
 sequence $T$ over
  $G$ of length $|T|\geq d$ contains a zero-sum subsequence of length $\ell$. For $\ell =\exp(G)$, we abbreviate $\mathsf s(G)=\mathsf s_{\exp(G)}(G)$ which is
  called   EGZ constant,  and for $\ell =|G|$ we abbreviate $\mathsf E(G)=\mathsf s_{|G|}(G)$, which is sometimes called Gao-constant (see \cite[page 193]{Grynkiewicz2013}). In 1996, the second author of this paper established a connection between
Erd\H{o}s-Ginzburg-Ziv Theorem and Davenport constant.

\noindent\textbf{Throrem B.}  (Gao, \cite{Gao1996,Gao2003}) \ {\sl
If $G$ is a finite abelian group, then $\mathsf s_{\ell}(G)
=\mathsf D(G)+\ell-1$ holds for all positive integer $\ell$ providing that
$\ell\geq |G|$ and $\exp(G)\mid \ell$. }

From Theorem B we know that $$\mathsf E(G)=\mathsf D(G)+|G|-1. \ \ \ \ \ (*)$$  This
formula has stimulated a lot of further research (see
\cite{AdhikariChen08,AdhikariChenFrKoPa06,AdhikariRath06,GH,
Grynkiewicz12, {Grynkiewicz-book},Luca07,YuanZeng10}
 for example). Among others, the full Chapter
16 in the recent monograph \cite{Grynkiewicz-book} is devoted to
this result. Indeed, the final Corollary in this chapter, (Corollary
16.1, page 260), provides a far-reaching generalization of the
initial formula (*), which is called the  $\Psi$-weighted Gao
Theorem. The formula $\mathsf E(G)=\mathsf D(G)+|G|-1$ has also been
 generalized to some  finite groups (not necessarily
 commutative, for example see \cite{Bass2007} and \cite{GL2010}).

In this paper, we aim to
 generalize $\mathsf E(G)=\mathsf D(G)+|G|-1$ to some  abstract finite commutative semigroups. To proceed,
we shall need some preliminaries. The notations on zero-sum theory
used in this paper are consistent with \cite{GaoGeroldingersurvey}
and notations on commutative semigroups are consistent with
\cite{Grillet monograph}. For sake of completeness, we introduce
some necessary ones.

\noindent $\bullet$  Throughout this paper, we shall always denote by $\mathcal{S}$ a finite commutative
semigroup and by $G$ a finite commutative group. We abbreviate ``finite commutative semigroup" into ``{\bf f.c.s.}".

Let ${\cal F}(\mathcal{S})$ be the (multiplicatively written) free
commutative monoid with basis $\mathcal{S}$. Then any $A\in {\cal
F}(\mathcal{S})$, say $A=a_1a_2\cdot\ldots\cdot a_n$, is a sequence
of elements in the semigroup $\mathcal{S}$. The identify element of
${\cal F}(\mathcal{S})$ is denoted by $1\in {\cal F}(\mathcal{S})$
(traditionally, the identity element is also called the empty
sequence). For any subset $S_0\subset \mathcal{S}$, let $A(S_0)$
denote the subsequence of $A$ consisting of all the terms from
$S_0$.

The operation of the semigroup $\mathcal{S}$ is denoted by ``$+$".
The identity element of $\mathcal{S}$, denoted $0_{\mathcal{S}}$ (if
exists), is the unique element $e$ of $\mathcal{S}$ such that
$e+a=a$ for every $a\in \mathcal{S}$. The zero element of
$\mathcal{S}$, denoted $\infty_{\mathcal{S}}$ (if exists), is the
unique element $z$ of $\mathcal{S}$ such that $z+a=z$ for every
$a\in \mathcal{S}$. Let
$$\sigma(A)=a_1+a_2+\cdots+a_n$$ be the sum of all terms in the
sequence $A$.  If $S$ has an identity element $0_{\mathcal{S}}$,  we
allow $A=1$, the empty sequence  and adopt the convention that
$\sigma(1)=0_\mathcal{S}$. Denote
$$\begin{array}{llll}\mathcal{S}^{0}=\left\{\begin{array}{llll}
               \mathcal{S},  & \mbox{if \ \ } \mathcal{S} \mbox{ has an identity element};\\
               \mathcal{S}\cup \{0\},  & \mbox{if \ \ } \mathcal{S} \mbox{ does not have an identity element.}\\
              \end{array}
              \right.
\end{array}$$
For any element $a\in \mathcal{S}$, the {\bf period} of $a$ is the
least positive integer $t$ such that $ra=(r+t)a$ for some integer
$r>0$. We define $\exp(\mathcal{S})$ to be the period of
$\mathcal{S}$, which is the least common multiple of all the periods
of the elements of $\mathcal{S}$. Let $A,B\in
\mathcal{F}(\mathcal{S})$ be sequences on $\mathcal{S}$. We call $B$
a {\sl proper subsequence} of $A$ if $B \mid A$ and   $B\neq A$. We
say that $A$ is {\it reducible} if $\sigma(B)=\sigma(A)$ for some
proper subsequence $B$ of $A$ (note that, $B$ is probably the empty
sequence $1$ if $\mathcal{S}$ has the identity element
$0_{\mathcal{S}}$ and $\sigma(A)=0_{\mathcal{S}})$. Otherwise, we
call $S$ {\it irreducible}.

\begin{definition}  \ Let $\mathcal{S}$ be an additively written commutative semigroup.
\begin{enumerate}
\item Let $\mathsf d(\mathcal{S})$ denote the smallest $\ell \in
\mathbb{N}_0\cup \{\infty\}$ with the property:

For any $m\in \mathbb{N}$ and $c_1, \ldots,c_m\in \mathcal{S}$ there
exists a subset $J \subset [1,m]$ such that $|J|\leq \ell$ and
$$
\sum_{j=1}^m c_j=\sum_{j \in J}c_j.
$$

\item Let $\mathsf D(\mathcal{S})$ denote the smallest $\ell \in
\mathbb{N}\cup \{\infty\}$ such that every sequence $A\in
\mathcal{F}(\mathcal{S})$ of length $|A|\geq \ell$ is reducible.

\item We call $\mathsf d(\mathcal{S})$ the small Davenport constant of $\mathcal{S}$,
and $\mathsf D(\mathcal{S})$ the (large) Davenport constant of
$\mathcal{S}$.
\end{enumerate}
\end{definition}

The small Davenport constant was introduced in \cite[Definition
2.8.12]{GH}, and  the large Davenport constant was first studied in
\cite{gaowang}. For convenience of the reader, we state the
following (well known) conclusion.

\begin{prop} \label{small-large} \ Let $\mathcal{S}$ be a finite  commutative semigroup. Then,
\begin{enumerate}
\item  $\mathsf d(\mathcal{S})<
\infty.$
\item  $\mathsf D(\mathcal{S})=\mathsf d(\mathcal{S})+1.$
\end{enumerate}
\end{prop}

\begin{proof} 1. See \cite[Proposition 2.8.13]{GH}.

2.  Take an arbitrary sequence $A\in \mathcal{F}(\mathcal{S})$ of
length at least $\mathsf d(\mathcal{S})+1.$ By the definition of
$\mathsf d(\mathcal{S})$, there exists a subsequence $A'$ of $A$
such that $|A'|\leq \mathsf d(\mathcal{S})<|A|$ and
$\sigma(A')=\sigma(A)$. This proves
$$
\mathsf D(\mathcal{S})\leq \mathsf d(\mathcal{S})+1.
$$
Now it remains to  prove $\mathsf d(\mathcal{S})\leq \mathsf
D(\mathcal{S})-1.$
 Take an arbitrary sequence $T\in \mathcal{F}(\mathcal{S})$. Let $T'$ be the minimal (in length) subsequence of $T$ such that
$\sigma(T')=\sigma(T)$. It follows that $T'$ is irreducible and
$|T'|\leq \mathsf D(\mathcal{S})-1$. By the arbitrariness of $T$, we
have
$$\mathsf d(\mathcal{S})\leq \mathsf D(\mathcal{S})-1.$$
\end{proof}

\begin{definition}
Define  $\mathsf E(\mathcal{S})$ of any f.c.s. $\mathcal{S}$ as the smallest
positive integer $\ell$ such that, every sequence $A\in
\mathcal{F}(\mathcal{S})$ of length $\ell$ contains a subsequence
$B$ with $\sigma(B)=\sigma(A)$ and $|A|-|B|=\kappa(\mathcal{S})$, where
\begin{equation}\label{kap}
\kappa(\mathcal{S})=\left\lceil\frac{|\mathcal{S}|}{\exp(\mathcal{S})}\right\rceil\exp(\mathcal{S}).
\end{equation}
\end{definition}

Note that if $\mathcal{S}=G$ is a finite abelian group, the
invariants $\mathsf E(\mathcal{S})$ and $\mathsf D(\mathcal{S})$
  are consistent  with the
classical  invariants $\mathsf D(G)$ and $\mathsf E(G)$,
respectively.  We suggest  the following generalization of $\mathsf
E(G)=\mathsf D(G)+|G|-1.$

\medskip

\begin{conj}\label{conjecture E(G)=D(G)+kappa(G)-1}
For any f.c.s. $\mathcal{S}$,
$$\mathsf E(\mathcal{S})\leq \mathsf D(\mathcal{S})+ \kappa(\mathcal{S})-1.$$
\end{conj}

We remark that $\mathsf E(\mathcal{S})\geq \mathsf D(\mathcal{S})+
\kappa(\mathcal{S})-1$, the converse of the above inequality, holds
trivially when $\mathcal{S}$ contains an identity element
$0_{\mathcal{S}}$, i.e., when $\mathcal{S}$ is a finite commutative
monoid. The extremal sequence can be obtained by any irreducible
sequence $T$ of length $\mathsf D(\mathcal{S})-1$ adjoined with a sequence
of length $\kappa(\mathcal{S})-1$ of all terms equaling
$0_{\mathcal{S}}$. So, Conjecture \ref{conjecture
E(G)=D(G)+kappa(G)-1}, if true, would imply the following.

\begin{conj}\label{conjecture E(G)=D(G)+kappa(G)-1(2)}
For any finite commutative monoid $\mathcal{S}$,
$$\mathsf E(\mathcal{S})= \mathsf D(\mathcal{S})+ \kappa(\mathcal{S})-1.$$
\end{conj}

Nevertheless, to make the study more general, the existence of the
identity element is not necessary. We shall verify that Conjecture
\ref{conjecture E(G)=D(G)+kappa(G)-1} holds for some important
f.c.s., including group-free f.c.s, elementary f.c.s, and
archimedean f.c.s with certain constraints.

\section{On group-free semigroups}

We begin this section with some definitions.

On a commutative semigroup $\mathcal{S}$ the Green's preorder, denoted $\leqq_{\mathcal{H}}$, is defined by
$$a \leqq_{\mathcal{H}} b\Leftrightarrow a=b+t$$ for some $t\in \mathcal{S}^{0}$. Green's congruence, denoted
$\mathcal{H}$, is a basic relation introduced by Green for semigroups which is defined by:
$$a \ \mathcal{H} \ b \Leftrightarrow a \ \leqq_{\mathcal{H}} \ b \mbox{ and } b \ \leqq_{\mathcal{H}} \ a.$$
For an element $a$ of $\mathcal{S}$,  let $H_a$ be the congruence class by $\mathcal{H}$ containing $a$.
We call a f.c.s. $\mathcal{S}$ {\bf group-free}, provided that all its
subgroups are trivial, equivalently, $\exp(\mathcal{S})=1$. The group-free f.c.s is fundamental for Semigroup Theory due to the following property.

\noindent \textbf{Property C.} \ (see \cite{Grillet monograph}, Proposition 2.4 of Chapter V) \  {\sl For
any f.c.s. $\mathcal{S}$, the quotient semigroup  $\mathcal{S}\diagup \mathcal{H}$ of $\mathcal{S}$ by
$\mathcal{H}$ is group-free.}

We first show that Conjecture \ref{conjecture E(G)=D(G)+kappa(G)-1} holds true for any group-free f.c.s.,
for which the following lemma will be necessary.

\begin{lemma}\label{Lemma Green equality} (See \cite{Grillet monograph}, Proposition 2.3 of Chapter V) \
For any group-free f.c.s. $\mathcal{S}$, the Green's congruence $\mathcal{H}$ is the equality on $\mathcal{S}$.
\end{lemma}

Now we are ready to give the following.

\begin{theorem}\label{Theorem group-free hold} For any group-free f.c.s. $\mathcal{S}$,  $$\mathsf E(\mathcal{S})\leq
\mathsf D(\mathcal{S})+ \kappa(\mathcal{S})-1.$$
\end{theorem}

\begin{proof} Take an arbitrary sequence $T\in \mathcal{F}(S)$ of length $\mathsf D(\mathcal{S})+
\kappa(\mathcal{S})-1$. By the definition of $\mathsf D(\mathcal{S})$, there exists a subsequence $T'$ of $T$
with $|T'|\leq \mathsf D(\mathcal{S})-1$ and $\sigma(T')=\sigma(T).$ Let $T''$ be a subsequence of $T$ containing
$T'$, i.e., $$T'\mid T'',$$ with length
\begin{equation}\label{equation length T''}
|T''|=\mathsf D(\mathcal{S})-1.
\end{equation}
Note that $T''$ is perhaps equal to $T'$ for example when $|T'|=\mathsf D(\mathcal{S})-1$.
We see that $$\sigma(T')=\sigma(T) \ \leqq_{\mathcal{H}} \ \sigma(T'') \ \leqq_{\mathcal{H}} \
\sigma(T'),$$ and thus  $\sigma(T'') \ \mathcal{H} \ \sigma(T)$. By Lemma \ref{Lemma Green equality},
we have $$\sigma(T'')=\sigma(T).$$ Combined with \eqref{equation length T''}, we have the theorem proved.
\end{proof}

\begin{definition} \ A commutative nilsemigroup $\mathcal{S}$ is a commutative semigroup  with a zero
element $\infty_{\mathcal{S}}$ in which every element $x$ is nilpotent, i.e., $nx=\infty_{\mathcal{S}}$
for some $n>0$.
\end{definition}

Since any finite commutative nilsemigroup is group-free, we have the following
immediate corollary of Theorem \ref{Theorem group-free hold}.

\begin{cor}\label{Corollary nilsemigroup} \  Let $\mathcal{S}$ be a finite commutative nilsemigroup.
Then  $\mathsf E(\mathcal{S})\leq \mathsf D(\mathcal{S})+ \kappa(\mathcal{S})-1$.
\end{cor}

In the rest of this paper, we need only to consider the case of $\exp(\mathcal{S})>1$. Two classes
 of important f.c.s., finite elementary semigroups and finite archimedean semigroups, will be our
emphasis as both these semigroups are basic components in two kinds of decompositions of semigroups,
namely, subdirect decompositions and semilattice decompositions, correspondingly. Both decompositions have
been the mainstay of Commutative Semigroup Theory for many years (see \cite{Grillet monograph}).

\section{On elementary semigroups}

With respect to the subdirect decompositions, Birkhoff in 1944 proved the following.

\noindent \textbf{Theorem D.} (\cite{Grillet monograph}, Theorem 1.4 of Chapter IV) \ {\sl Every commutative semigroup is a subdirect product of subdirectly irreducible commutative semigroups.}

Hence, we shall give the following result with respect to the subdirect decomposition.

\begin{theorem}\label{Theorem subdirectly irreducible holds} For  any subdirectly
irreducible f.c.s. $\mathcal{S}$,  $$\mathsf E(\mathcal{S})\leq \mathsf D(\mathcal{S})+ \kappa(\mathcal{S})-1.$$
\end{theorem}

To prove Theorem \ref{Theorem subdirectly irreducible holds}, several preliminaries will be necessary.

\begin{lemma}\cite{Grillet77} \ \label{Lemma structure subdirect decomposition}  Any subdirectly irreducible
f.c.s.  is either a nilsemigroup, or an abelian group, or an elementary  semigroup. In particular, any f.c.s.
is a subdirect product of a commutative nilsemigroup, an abelian group, and several elementary semigroups.
\end{lemma}

\begin{definition}\label{Definition elementary semigroup} A commutative semigroup $\mathcal{S}$ is elementary
in case it is the disjoint union $\mathcal{S}=G\cup N$ of a group $G$ and a nilsemigroup $N$, in which $N$ is
an ideal of $\mathcal{S}$, the zero element $\infty_N$ of $N$ is the zero element $\infty_{\mathcal{S}}$
of $\mathcal{S}$ and the identity element of $G$ is the identity element of $S$.
\end{definition}

\begin{lemma}\label{Lemma pure congruence} (\cite{Grillet monograph}, Proposition 3.2 of Chapter IV) \
On any commutative nilsemigroup $N$, the relation $\mathcal{P}_N$ on $N$ given by
$a \ \mathcal{P}_N \ b \Leftrightarrow \infty:a=\infty:b$ is a congruence on $N$ with $\{\infty_{N}\}$ being a
 $\mathcal{P}_N$ class, where $\infty:a=\{x\in N^{0}: x+a=\infty_{N}\}$.
\end{lemma}

\begin{lemma} (\cite{Grillet monograph}, Proposition 5.1 of Chapter IV) \ \label{Lemma action on class}
In an elementary semigroup $\mathcal{S}=G\cup N$, the action of every $g\in G$ on $N$ permutes every
$\mathcal{P}_N$-class.
\end{lemma}

\begin{lemma}\label{Lemma addition in nilsemigroup} \ Let $N$ be a finite commutative nilsemigroup,
and let $a,b$ be two elements in $N$. If $a+b=a$ then $a=\infty_N$.
 \end{lemma}
\begin{proof} It is easy to see that $a=a+b=a+2b=\cdots=a+nb=\infty_N$ for some $n\in \mathbb{N}$, done.
 \end{proof}

\begin{lemma}\label{Lemma anihilator}
Let $N$ be a finite commutative nilsemigroup, and let $a,b$ be two elements of $N$ with $a<_{\mathcal{H}} b$.
Then $\infty:b \subsetneqq\infty:a$.
\end{lemma}

\begin{proof} The conclusion $\infty:b \subseteq\infty:a$ is clear. Hence, we need only to show that
$\infty:b \neq \infty:a$.
If $a=\infty_N$, then $\infty:a=N^0 \neq N \supseteq \infty:b$, done. Hence, we assume
$\infty_N<_{\mathcal{H}} a$. It follows that there exists some {\bf minimal} element $c$ of $N$ such that
$$\infty_N<_{\mathcal{H}} c \leqq_{\mathcal{H}} a.$$ Then there exists some $x\in N^{0}$ such that
$$a+x=c.$$ Since $a<_{\mathcal{H}} b$, then there exists some $y\in N$ such that $$b+y=a.$$
Since $c\neq \infty_N$, by Lemma \ref{Lemma addition in
nilsemigroup} we have $c+y <_{\mathcal{H}} c$ and hence combined
with the minimality of $c$, we have $ c+y =\infty_N $. Therefore, it
follows that $a+(x+y)=(a+x)+y=c+y=\infty_N$ and
$b+(x+y)=(b+y)+x=a+x=c$, and thus, $x+y\in \infty:a$ and $x+y\notin
\infty:b$, and we are done. \end{proof}

 \noindent {\bf Proof of Theorem \ref{Theorem subdirectly irreducible holds}.} \
By Lemma \ref{Lemma structure subdirect decomposition}, Theorem B and  Corollary \ref{Corollary nilsemigroup},
it suffices to consider the case of $\mathcal{S}=G\cup N$ is an elementary semigroup.

 Take an arbitrary sequence $T\in \mathcal{F}(\mathcal{S})$ of length
$\mathsf D(\mathcal{S})+ \kappa(\mathcal{S})-1$. Let $$T_1=T(G)$$ be the
subsequence of $T$ consisting of all the terms from $G$, and let
$$T_2=T(N)$$ be the subsequence of $T$ consisting of all the terms
from $N$. Then $T_1$ and $T_2$ are disjoint and $T_1\cdot T_2=T$.

 If $T_2=1$, the empty sequence, then $T=T_1$ is a sequence of elements in the subsemigroup  $G$ (group).
Since $\mathsf D(\mathcal{S})\geq \mathsf D(G)$
 and $\kappa(\mathcal{S})\geq |\mathcal{S}|\geq |G|$ is a multiple of $\exp(G)=\exp(\mathcal{S})$,
it follows from Theorem B that there exists a subsequence
 $T'$ of $T$ with $$|T'|=\kappa(\mathcal{S})$$ and $$\sigma(T')=0_G=0_{\mathcal{S}}.$$
 Let $$T''=T\cdot T'^{-1}.$$ We see that $\sigma(T'')=\sigma(T'')+0_{\mathcal{S}}=\sigma(T'')+\sigma(T')
=\sigma(T)$, and we are done. Hence, we need only to consider the case that
 \begin{equation}\label{equation T2 neq empty}
 T_2\neq 1.
 \end{equation}

 Assume $\sigma(T)=\infty_{\mathcal{S}}$. Then there exists a subsequence $U$ of $T$ with
$$\sigma(U)=\sigma(T)$$ and $|U|\leq \mathsf D(\mathcal{S})-1$.  Take a subsequence $U'$ of $T$ with $$U\mid U'$$
and $$|U'|=\mathsf D(\mathcal{S})-1.$$
We check that $\sigma(U')=\sigma(U)+\sigma(U'U^{-1})
=\infty_{\mathcal{S}}+\sigma(U'U^{-1})=\infty_{\mathcal{S}}=\sigma(T)$, and we are done.
Hence, we need only to consider the case that
 \begin{equation}\label{equation sigma(T) infty}
 \sigma(T)\neq \infty_{\mathcal{S}}.
 \end{equation}

Define the subgroup $$K=\{g\in G: g+\sigma(T_2)=\sigma(T_2)\}$$ of $G$, i.e., $K$ is the stabilizer of
$\sigma(T_2)$ in $G$ when considering the action of $G$ on $N$. We claim that
 \begin{equation}\label{equation D(S)geq |T2|+D()}
 \mathsf D(\mathcal{S})\geq |T_2|+\mathsf D(G\diagup K).
 \end{equation}
 Take a sequence $W\in \mathcal{F}(G)\subset \mathcal{F}(\mathcal{S})$ such that
 $\varphi_{G\diagup K}(W)$ is zero-sum free in the quotient group $G\diagup K$ with $$|W|=\mathsf D(G\diagup K)-1,$$
where $\varphi_{G\diagup K}$ denotes the canonical epimorphism of $G$ onto $G\diagup K$.
 To prove \eqref{equation D(S)geq |T2|+D()}, it suffices to verify that $W\cdot T_2$ is irreducible
in $\mathcal{S}$. Suppose to the contrary that $W\cdot T_2$ contains a proper subsequence $V$ with
 \begin{equation}\label{equation sigma(V)=sigma()}
 \sigma(V)=\sigma(W\cdot T_2).
  \end{equation}
  Let
$$V=V_1\cdot V_2$$
   with
   $$V_1\mid W$$ and
   \begin{equation}\label{equation V2 mid T2}
   V_2\mid T_2.
    \end{equation}
    By \eqref{equation T2 neq empty}, we have $\sigma(W\cdot T_2)\in N$, which implies $$V_2\neq 1.$$
By Lemma \ref{Lemma action on class}, we have that
  $\sigma(V_2) \ \mathcal{P}_N \ \sigma(V)$ and $\sigma(T_2) \ \mathcal{P}_N \ \sigma(W\cdot T_2)$.
Combined with \eqref{equation sigma(V)=sigma()}, we have that \begin{equation}\label{equation sigma(V2)Psigma(T2)}
  \sigma(V_2) \ \mathcal{P}_N \ \sigma(T_2).
   \end{equation}
   By \eqref{equation sigma(T) infty}, we have
   \begin{equation}\label{equation sigma(T2)neqinftyN}
   \sigma(T_2)\neq \infty_N.
 \end{equation}
     By \eqref{equation V2 mid T2}, we have $$\sigma(T_2) \ \leqq_{\mathcal{H}} \ \sigma(V_2),$$ where
$\leqq_{\mathcal{H}}$ denotes the Green's preorder in the nilsemigroup $N$. Combined with \eqref{equation
sigma(V2)Psigma(T2)} and Lemma \ref{Lemma anihilator}, we derive that $\sigma(T_2) \ \mathcal{H} \ \sigma(V_2).$
It follows from
Lemma \ref{Lemma Green equality} that $$\sigma(T_2)=\sigma(V_2).$$
Combined with \eqref{equation V2 mid T2}, \eqref{equation sigma(T2)neqinftyN} and Lemma \ref{Lemma addition
in nilsemigroup}, we conclude that
$$V_2=T_2.$$
Recalling that $V$ is a proper subsequence of $W\cdot T_2$, we have $V_1\neq W$. Since
$\varphi_{G\diagup K}(W)$ is zero-sum free in the group $G\diagup K$, we derive that
$\sigma(V_1)-\sigma(W)\notin K$, and thus,
$\sigma(V)=\sigma(V_1)+\sigma(V_2)=\sigma(V_1)+\sigma(T_2)\neq \sigma(W)+\sigma(T_2)=\sigma(W\cdot T_2)$,
 a contradiction with \eqref{equation sigma(V)=sigma()}. This proves \eqref{equation D(S)geq |T2|+D()}.

By \eqref{equation D(S)geq |T2|+D()}, we have that
$$\begin{array}{llll}
|T_1|&=&|T|-|T_2|\\
&\geq& \mathsf D(\mathcal{S})+\kappa(\mathcal{S})-1-
 (\mathsf D(\mathcal{S})-\mathsf D(G\diagup K))\\
&=&\mathsf D(G\diagup K)+\kappa(\mathcal{S})-1.\\
\end{array}$$
Applying Theorem B, we derive that $T_1$ contains a subsequence $T_1'$
with $$|T_1'|=\kappa(\mathcal{S})$$ such that $\varphi_{G\diagup K}(\sigma(T_1'))=0_{G\diagup K}$, i.e.,
\begin{equation}\label{equation sigma(T1')inK}
\sigma(T_1')\in K.
\end{equation}
Let $T_2'=T\cdot T_1'^{-1}$. Observe $T_2\mid T_2'$. By \eqref{equation sigma(T1')inK}, we  check that
$$\begin{array}{llll}
\sigma(T)&=&\sigma(T_1')+\sigma(T_2')\\
&=& \sigma(T_1')+(\sigma(T_2)+\sigma(T_2'\cdot T_2^{-1}))\\
&=&(\sigma(T_1')+\sigma(T_2))+\sigma(T_2'\cdot T_2^{-1})\\
&=&\sigma(T_2)+\sigma(T_2'\cdot T_2^{-1})\\
&=&\sigma(T_2').\\
\end{array}$$
This completes the proof of the theorem. \qed

We remark that since the elementary semigroup $\mathcal{S}=G\cup N$ has an identity element
$0_\mathcal{S}=0_G$, the equality in the above theorem holds as noted in the introductory section,
i.e., $\mathsf E(\mathcal{S})=\mathsf D(\mathcal{S})+\kappa(S)-1$.

\section{On archimedean semigroups}

In this section, we shall deal with the
class of semigroups associated to semilattice decomposition of semigroups. The semilattice decompositions
were obtained for f.c.s by Schwarz \cite{Schwarz53} and Thierrin \cite{Thierrin54}, and then were extended
to all commutative semigroups by Tamura and Kimura \cite{TamuraKimura54}, in which they proved the following.

\noindent \textbf{Theorem E.} \ {\sl Every commutative semigroup is a semilattice of commutative archimedean
semigroups.}

\begin{definition} A commutative semigroup $\mathcal{S}$ is called archimedean provided that for any two
elements $a,b\in \mathcal{S}$, there exist $m,n>0$ and $x,y\in \mathcal{S}$ with $ma=b+x$ and $nb=a+y$.
\end{definition}

To be precise, for any commutative semigroup $\mathcal{S}$ there exists a semilattice $Y$ and a partition
$\mathcal{S}=\bigcup_{a\in Y} \mathcal{S}_a$ into subsemigroups $\mathcal{S}_a$ (one for every $a\in Y$)
with $\mathcal{S}_a+\mathcal{S}_b\subseteq \mathcal{S}_{a\wedge b}$ for all $a,b\in Y$, and moreover, each
component
$\mathcal{S}_a$ is archimedean. Hence, we shall consider Conjecture \ref{conjecture E(G)=D(G)+kappa(G)-1} on
archimedean semigroups in what follows. To proceed with it, several preliminaries will be necessary.

\begin{definition} We call a commutative semigroup $\mathcal{S}$ nilpotent if
$|\underbrace{\mathcal{S}+\cdots+\mathcal{S}}\limits_{t}|=1$ for some $t>0$. For any commutative nilpotent
semigroup $\mathcal{S}$, the least such positive integer $t$ is called the nilpotency index and is denoted by
$\mathcal{L}(\mathcal{S})$.
\end{definition}

Note that, when the commutative semigroup $\mathcal{S}$ is finite, $\mathcal{S}$ is nilpotent if and only if $\mathcal{S}$ is a nilsemigroup. With respect to finite semigroups, the famous Kleitman-Rothschild-Spencer conjecture (see \cite{KRSconjeture76}) states that, on a statistical
basis, almost all finite semigroups are nilpotent of index at most three, for which there is considerable
evidence, but gaps in the original proof have remained unfilled. For the commutative version of this conjecture,
there is also some evidence.

We need to give some important notions, namely the Rees congruence and the Rees quotient.
Let $\mathcal{I}$ be an ideal of a commutative semigroup $\mathcal{S}$.
The relation $\mathcal{J}$ defined by $$a\ \mathcal{J} \ b\Leftrightarrow a=b \mbox{ or } a,b\in \mathcal{I}$$
is a congruence on $\mathcal{S}$, the \textbf{Rees congruence} of the ideal $\mathcal{I}$.
Let $\mathcal{S}\diagup \mathcal{I}$ denote the quotient semigroup $\mathcal{S}\diagup \mathcal{J}$,
which is called the \textbf{Rees quotient semigroup} of $\mathcal{S}$ by $\mathcal{I}$. The Rees congruence
and the resulting Rees quotient semigroup introduced by Rees \cite{Rees40} in 1940 have been
 among the basic notions in Semigroup Theory.
In some sense, the Rees quotient semigroup is obtained by squeezing $\mathcal{I}$ to a zero element
(if $\mathcal{I}\neq \emptyset$) and leaving $S\setminus I$ untouched. Hence, it is not hard to obtain the
following lemma.

\begin{lemma}\label{Lemma Davenport Rees} \ For any ideal $\mathcal{I}$ of a f.c.s.
 $\mathcal{S}$, $$\mathsf D(\mathcal{S})\geq \mathsf D(\mathcal{S}\diagup \mathcal{I}).$$
\end{lemma}

 \begin{lemma}\label{lemma structure of Archimedean} (\cite{Grillet monograph}, Chapter III,
Proposition 3.1) \ A commutative semigroup $S$ which contains an idempotent $e$ (for instance
a f.c.s.) is archimedean if and only if it is an ideal extension of an abelian group $G$ by a
commutative nilsemigroup $N$; then $S$ has a kernel $K=H_e=e+S$ and $S\diagup K$ is a commutative nilsemigroup.
 \end{lemma}

  \begin{lemma}\label{Lemma leq D(N)leq} \ For any finite commutative nilsemigroup $N$,
  $$\mathcal{L}(N)\leq \mathsf D(N) \leq \mathcal{L}(N)+1.$$
 \end{lemma}

 \begin{proof} By the definition of $\mathcal{L}(N)$, there exists a sequence
$T\in \mathcal{F}(N)$ with $|T|=\mathcal{L}(N)-1$
and $\sigma(T)\neq \infty_N$. By Lemma \ref{Lemma addition in nilsemigroup}, we have that $T$ is
irreducible, which implies $\mathsf D(N)\geq |T|+1=\mathcal{L}(N)$. On the other hand, since any sequence
in $\mathcal{F}(N)$ of length  $\mathcal{L}(N)$ has a sum $\infty_N$, we have $\mathsf D(N) \leq \mathcal{L}(N)+1$,
and we are through.
\end{proof}

Now we are in a position to put out our result on archimedean semigroup as follows.

\begin{theorem} \label{Theorem Archimedean E(G)} For any  finite archimedean semigroup
$\mathcal{S}$, $\mathsf E(\mathcal{S})\leq \mathsf D(\mathcal{S})+\kappa(\mathcal{S})$. Moreover, if the
nilsemigroup $\mathcal{S}\diagup K$ has a nilpotency index at most three, then
$\mathsf E(\mathcal{S})\leq \mathsf D(\mathcal{S})+\kappa(\mathcal{S})-1$, where $K$ denotes the kernel of $\mathcal{S}$.
\end{theorem}

\begin{proof}  Let $e$ be the unique idempotent of $\mathcal{S}$.
By Lemma \ref{lemma structure of Archimedean}, we have that the kernel $K=H_e=e+S$ and the Rees quotient semigroup
\begin{equation}\label{equation N=S/K}
N=\mathcal{S}\diagup K
\end{equation} is a nilsemigroup. By Theorem B and Corollary \ref{Corollary nilsemigroup}, we need only to
consider the case that both $K$ and $N$ are nontrivial. We claim that
\begin{equation}\label{equation D(G)>=D(N),D(K)+1}
\mathsf D(\mathcal{S})\geq \max(\mathsf D(N), \ \mathsf D(K)+1).
\end{equation}
 $\mathsf D(\mathcal{S})\geq \mathsf D(N)$ follows from \eqref{equation N=S/K} and Lemma \ref{Lemma Davenport Rees}. Now take
a minimal zero-sum sequence $U$ of elements in the group $K$ (the kernel of $\mathcal{S}$) with length
$|U|=\mathsf D(K)$.
Since $N$ is nontrivial, the semigroup $\mathcal{S}$ has no identity element, which implies that $U$ is
irreducible in $\mathcal{S}$, and thus, $\mathsf D(\mathcal{S})\geq |U|+1=\mathsf D(K)+1$. This proves
\eqref{equation D(G)>=D(N),D(K)+1}.

Now we take a sequence $T\in \mathcal{F}(\mathcal{S})$ with  $|T|=\mathsf D(\mathcal{S})+\kappa(\mathcal{S})-\epsilon$,
where $\epsilon=0$ or $\epsilon=1$ according to the conclusions to prove in what follows. Since
$|T|=\mathsf D(\mathcal{S})+\kappa(\mathcal{S})-\epsilon\geq \mathsf D(N)+\kappa(G)-\epsilon\geq  \mathsf D(N)$, it follows
from Lemma \ref{Lemma leq D(N)leq} that $\sigma(T)$ belongs to the kernel of $\mathcal{S}$, i.e.,
\begin{equation}\label{equation sigma(T) in K}
\sigma(T)\in K.
\end{equation}
Let $\psi_K: \mathcal{S}\rightarrow K$ be the canonical retraction of $\mathcal{S}$ onto $K$, i.e.,
$$\psi_K(a)=e+a$$ for every $a\in \mathcal{S}$.
Notice that $\psi_K(T)$ is a sequence of elements in the kernel $K$ with length
$$|\psi_K(T)|=|T|=\mathsf D(\mathcal{S})+\kappa(\mathcal{S})-\epsilon\geq \mathsf D(K)+\kappa(\mathcal{S}).$$
Since $|\kappa(\mathcal{S})|\geq |K|$ and
$\exp(K)=\exp(\mathcal{S})\mid \kappa(\mathcal{S})$, it follows from Theorem B that
there exists a subsequence $T'$ of $T$ with $$|T'|=\kappa(\mathcal{S})$$ such that
 $\psi_K(T')$ is a zero-sum sequence in the kernel $K$, i.e.,
 \begin{equation}\label{equation psik()=e}
 \psi_K(\sigma(T'))=\sigma(\psi_K(T'))=e.
 \end{equation}
Now we assert the following.

\noindent \textbf{Claim.} \  If $\mathsf D(\mathcal{S})-\epsilon\geq \mathcal{L}(N)$ then $\mathsf E(\mathcal{S})
\leq \mathsf D(\mathcal{S})+\kappa(\mathcal{S})-\epsilon$.

 Since $|TT'^{-1}|=\mathsf D(\mathcal{S})-\epsilon\geq \mathcal{L}(N)$, it follows that $$\sigma(TT'^{-1})\in K.$$
Combined with \eqref{equation sigma(T) in K} and \eqref{equation psik()=e}, we have that
$$\begin{array}{llll}
\sigma(TT'^{-1})&=&\sigma(TT'^{-1})+e\\
&=& \sigma(TT'^{-1})+\psi_K(\sigma(T'))\\
&=&\sigma(TT'^{-1})+(e+\sigma(T'))\\
&=&(\sigma(TT'^{-1})+\sigma(T'))+e\\
&=&\sigma(T)+e\\
&=&\sigma(T).\\
\end{array}$$
Recall $|T'|=\kappa(\mathcal{S})$. This proves the claim.

By \eqref{equation D(G)>=D(N),D(K)+1}, Lemma \ref{Lemma leq D(N)leq}, and applying the above claim with
$\epsilon=0$, we conclude that $$\mathsf E(\mathcal{S})\leq \mathsf D(\mathcal{S})+\kappa(\mathcal{S}).$$
It remains to show $\mathsf E(\mathcal{S})\leq \mathsf D(\mathcal{S})+\kappa(\mathcal{S})-1$ when
$\mathcal{S}\diagup K$ has a nilpotency index at most three. Take $\epsilon=1$. By
 the above claim, we may assume without loss of generality that
$\mathsf D(\mathcal{S})\leq \mathcal{L}(N).$ Since $K$ is nontrivial, we have $\mathsf D(K)\geq 2.$ Combined with
$\mathcal{L}(N)\leq 3$ and \eqref{equation D(G)>=D(N),D(K)+1} and Lemma \ref{Lemma leq D(N)leq},
we conclude that $$\mathsf D(\mathcal{S})=\mathcal{L}(N)=3,$$
and $\mathsf D(K)=2$ which implies $$K=C_2,$$ the group of two elements.
Take a subsequence $T''$ of $T$ with length $$|T''|\leq \mathsf D(\mathcal{S})-1$$ and $$\sigma(T'')=\sigma(T).$$
If $|T''|=\mathsf D(\mathcal{S})-1$, we are done. Now assume $|T''|<\mathsf D(\mathcal{S})-1$, equivalently,
$$|TT''^{-1}|>\kappa(\mathcal{S}).$$ By \eqref{equation sigma(T) in K}, we have
\begin{equation}\label{equation sigma(T'')inK}
\sigma(T'')\in K,
\end{equation}
 and thus, $\sigma(\psi_K(TT''^{-1}))=\psi_K(\sigma(TT''^{-1}))=e+\sigma(TT''^{-1})=e$.
Since $\exp(\mathcal{S})=\exp(K)=2$, we can find a subsequence $U$ of $TT''^{-1}$ of length
exactly $\kappa(\mathcal{S})$ such that $$\sigma(\psi_K(U))=e.$$ Since $T''\mid TU^{-1}$, it follows from
\eqref{equation sigma(T'')inK} that $\sigma(TU^{-1})\in K$, and thus, $$\begin{array}{llll}
\sigma(TU^{-1})&=&\sigma(TU^{-1})+e\\
&=&\sigma(TU^{-1})+\sigma(\psi_K(U))\\
&=&\sigma(TU^{-1})+\psi_K(\sigma(U))\\
&=&\sigma(TU^{-1})+e+\sigma(U)\\
&=&\sigma(T)+e\\
&=&\sigma(T).\\
\end{array}$$ This completes the proof of the theorem.
\end{proof}

\bigskip

\noindent {\bf Acknowledgements}

\noindent  This work is supported by NSFC (11301381, 11271207,
11001035), Science and Technology Development Fund of Tianjin Higher
Institutions (20121003). This work was initiated during  the  first
author visited the Center for Combinatorics of Nankai University in
2010, he would like to thank the host's hospitality. In addition, the authors would like to thank the referee for many valuable comments and suggestions.

\end{document}